\documentclass[10pt,reqno]{amsart}
\usepackage{bbm}
\usepackage{a4wide}
\usepackage{dsfont}            
\usepackage{amssymb,amsmath,amsfonts}
\usepackage{mathrsfs}
\usepackage{graphicx} 
\usepackage{xcolor}          
\usepackage[utf8]{inputenc}
\usepackage[all]{xy}

\usepackage{url}

\usepackage[capitalise]{cleveref}

\crefname{section}{Section}{Sections}
\crefformat{section}{#2Section~#1#3} 
\Crefformat{section}{#2Section~#1#3} 

\crefname{subsection}{\S}{\S\S}
\crefformat{subsection}{#2\S#1#3} 
\Crefformat{subsection}{#2\S#1#3} 

\crefname{definition}{definition}{definitions}
\crefformat{definition}{#2definition~#1#3} 
\Crefformat{definition}{#2Definition~#1#3} 

\crefname{ex}{example}{examples}
\crefformat{example}{#2example~#1#3} 
\Crefformat{example}{#2Example~#1#3} 

\crefname{remark}{remark}{remarks}
\crefformat{remark}{#2remark~#1#3} 
\Crefformat{remark}{#2Remark~#1#3} 

\crefname{convention}{convention}{conventions}
\crefformat{convention}{#2convention~#1#3} 
\Crefformat{convention}{#2Convention~#1#3}

\crefname{lemma}{lemma}{lemmas}
\crefformat{lemma}{#2lemma~#1#3} 
\Crefformat{lemma}{#2Lemma~#1#3} 

\crefformat{proposition}{#2proposition~#1#3} 
\Crefformat{proposition}{#2Proposition~#1#3} 

\Crefformat{corollary}{#2Corollary~#1#3}

\crefname{theorem}{theorem}{theorems}
\crefformat{theorem}{#2theorem~#1#3} 
\Crefformat{theorem}{#2Theorem~#1#3} 

\crefname{assumption}{assumption}{Assumptions}
\crefformat{assumption}{#2assumption~#1#3} 
\Crefformat{assumption}{#2Assumption~#1#3} 

\crefname{equation}{}{}
\crefformat{equation}{(#2#1#3)} 
\Crefformat{equation}{(#2#1#3)}

\newtheorem{proposition}{Proposition}[section]
  \newtheorem{theorem}[proposition]{Theorem}
  \newtheorem{corollary}[proposition]{Corollary}
  \newtheorem{lemma}[proposition]{Lemma}
    
\theoremstyle{definition}
  \newtheorem{definition}[proposition]{Definition}
  \newtheorem{remark}[proposition]{Remark}
   
  \newtheorem{example}[proposition]{Example}
   \newtheorem{problem}[proposition]{Problem}
\allowdisplaybreaks

\newcommand{\cst}{\ifmmode\mathrm{C}^*\else{$\mathrm{C}^*$}\fi}

\newcommand\cA{\mathcal A}
\newcommand\cB{\mathcal B}
\newcommand\cC{\mathcal C}
\newcommand\cD{\mathcal D}

\newcommand\cH{\mathcal H}

\newcommand\cN{\mathcal N}

\newcommand\cV{\mathcal V}
\newcommand\cW{\mathcal W}

\newcommand{\kk}{\mathbbm{k}}

\newcommand{\CC}{\mathbb{C}}

\newcommand{\id}{\mathrm{id}}

\newcommand{\I}{\mathds{1}}

\newcommand{\ww}{\mathrm{W}}

\numberwithin{equation}{section}

\def\labelitemi{$\blacktriangleright$}

\author{Pawe{\l} Kasprzak}
\address{Department of Mathematical Methods in Physics, Faculty of Physics, University of Warsaw, Poland}
\email{pawel.kasprzak@fuw.edu.pl}
 
\title[Generalized (co)integrals]{Generalized (co)integrals on coideal subalgebras}



\begin{document}
\begin{abstract}
Given a  a Hopf algebra  $\cH$,   a left coideal subalgebra $\cA\subset \cH$ and a non-zero multiplicative functional $\mu\in\cA'$, we consider the space of  left    $\mu$-integrals $L^\cA_\mu\subset\cA$, where $\Lambda\in L^\cA_\mu$ if $a\Lambda= \mu(a)\Lambda$ for all $a\in\cA$. We observe that $\dim L^\cA_\mu=1$ if $\cA$ is a Frobenius algebra and we  conclude this equality   for finite dimensional left coideal subalgebras of a weakly finite Hopf algebra. In general we prove that if $\dim L^\cA_\mu>0$ then $\dim\cA<\infty$.  Given a group-like element $g\in\cH$ we  consider  the space $L^g_{\cA}\subset\cA'$ of $g$-cointegrals on $\cA$, where $\phi\in L^g_{\cA}$ if $(\id\otimes\phi)(\Delta(a)) = g\phi(a)$ for all $a\in \cA$. We prove that if $\cA$ admits a non-degenerate $\mu$-integral then $\dim L^g_{\cA}\in\{0,1\}$. Linking further both concepts we prove that:
\begin{itemize}
    \item every semisimple left coideal subalgebra  $\cA\subset \cH$ which is preserved by the antipode squared admits a  faithful $\I$-cointegral;
    \item every unimodular finite dimensional left coideal subalgebra $\cA\subset \cH$ admitting a faithful $\I$-cointegral is preserved by the antipode squared;
    \item every non-degenerate right group-like projection in a cosemisimple Hopf algebra is a two sided group-like projection.
\end{itemize} 
Finally we list all $\varepsilon$-integrals for left  coideals subalgebras in Taft algebras and we  list all $g$-cointegrals on them. 
\end{abstract}

\maketitle

\setcounter{tocdepth}{1}

\newlength{\sw}
\settowidth{\sw}{$\scriptstyle\sigma-\text{\rm{weak closure}}$}
\newlength{\nc}
\settowidth{\nc}{$\scriptstyle\text{\rm{norm closure}}$}
\newlength{\ssw}
\settowidth{\ssw}{$\scriptscriptstyle\sigma-\text{\rm{weak closure}}$}
\newlength{\snc}
\settowidth{\snc}{$\scriptscriptstyle\text{\rm{norm closure}}$}
\renewcommand{\labelitemi}{$\bullet$}
\section{Introduction}
The theory of finite dimensional Hopf algebras is centered around the concept of integrals  (c.f. e.g.  \cite{Lar_Sweed}, \cite{Radf_194}, \cite{Sull}, \cite{Sweed}, \cite{VD}). Integrals for Hopf algebras were first considered by Larson and Sweedler in \cite{Lar_Sweed} where they proved their existence and uniqueness in the finite dimensional case. Our main reference to the subject of Hopf algebras and (co)integrals on them will be  \cite{Radford_book}, and \cite[Section 10]{Radford_book} especially, where the theory of left integrals assigned to a non-zero multiplicative functional $\mu\in\cH'$  is fully developed. Defining \[L^\cH_\mu = \{\Lambda\in\cH:a\Lambda = \mu(a)\Lambda\}\] it turns out that  we have $\dim L^\cH_\mu = 1$ if and only if $\dim\cH < \infty$ and $L^\cH_\mu =\{0\}$ otherwise, see \cite[Proposition 10.2.1.]{Radford_book} and \cite[Proposition 10.6.2]{Radford_book}. Moreover if $\Lambda\neq 0$, then  the map $\cH'\ni\mu\mapsto\Lambda*\mu:=(\mu\otimes\id)\Delta(\Lambda)\in\cH$ is a right $\cH'$-modules isomorphism, c.f. \cite[Theorem 10.2.2]{Radford_book} (here $\cH'$ stands for the dual of $\cH$). 

The dual to    $\mu$-integral is the concept of  $g$-cointegral on a Hopf algebra $\cH$: given a group-like element $g\in\cH$ we define $L^g_{\cH}\subset\cH'$, where
\[L^g_{\cH} = \{\phi\in\cH':(\id\otimes\phi)((\Delta(a)) = g\phi(a)\textrm{ for all } a\in\cH\}.\] 
Using \cite[Theorem 10.9.5]{Radford_book}, \cite[Theorem 10.9.2]{Radford_book} we see that $\dim L^g_{\cH}\in\{0,1\}$. Moreover if $\phi\in L^g_{\cH}$ is non-zero then it is faithful. Note that  it may happen that $\dim L^g_{\cH} = 1$ for $\dim\cH=\infty$. Indeed this is the case for each inifinite dimensional cosemisimple Hopf algebra (for the conditions equivalent to $\dim L^g_{\cH} = 1$ see \cite[Theorem 10.9.2]{Radford_book}). 

Let $\cA$ be a left coideal subalgebra of $\cH$, $\mu\in\cA'$ a non-zero multiplicative functional on $\cA$ and $g\in\cH$ a group-like element.
The  concepts of  $\mu$-integrals and $g$-cointegrals still makes sense in this context and hence the  notation  $L^\cA_\mu\subset \cA, L^g_{\cA}\subset \cA'$ will be used throughout the paper. Our study  was initiated   in \cite{Chir_Kasp_Szu}, where the case of  $\varepsilon$-integrals $\Lambda\in L_\varepsilon^\cA$ satisfying $\Lambda^2 = \Lambda$  and $\Lambda a = \Lambda\varepsilon(a)$ for all $a\in\cA$  were analyzed thoroughly. They were linked with the theory of right  group-like projections in $\cH$ and it was proved in \cite{Chir_Kasp_Szu} that $\dim \cA<\infty$ if $\cA$  admits a non-zero right group like projection $P\in L_\varepsilon^\cA$. Moreover the existence of non-zero integrals   was proved for semisimple $\cA$ and the converse result was obtained at least when $S^2(\cA) = \cA$ (see \cite[Corollary 3.26]{Chir_Kasp_Szu},  \cite[Theorem 3.27]{Chir_Kasp_Szu}).  
In the present  paper we generalize  the finite dimensionality result, showing that if $\dim L^\cA_\mu>0$ then $\dim\cA<\infty$, c.f. \Cref{thm_fd}. Moreover we observe that if $\cA$ is a Frobenius algebra (e.g. when $\dim\cH<\infty$, c.f. \cite[Theorem 6.1]{skr-proj}) then $\dim L^\cA_\mu = 1$, see \Cref{thm_dim}. The non-degeneracy  of $\Lambda$ is also a matter of our study,  where we say that  $\Lambda\in L_\mu^\cA$ is non-degenerate  if the map  
\[\cH'\ni\mu\mapsto \Lambda*\mu\in\cA\] is onto. We prove this to hold if $\cA$ is Frobienius, see \Cref{biann_non}.   The question of non-degeneracy of an arbitrary $\mu$-integral is left open.

The concepts of integrals and conitegrals on left coideals subalgebras turn  out to be linked  intimately. 
In particular we observe  that if $\cA\subset \cH$  admits a non-degenerate $\mu$-integral then $\dim L^g_\cA\in\{0,1\}$ and if $\phi\in L^g_\cA$ is non-zero then it is faithful, c.f. \Cref{thm_gcoin} and \Cref{cor_frob}.  Moreover we prove that (see \Cref{uni_cor})
   \begin{itemize}
    \item every semisimple left coideal subalgebra of $\cH$ which is preserved by the antipode squared admits an essentially  unique faithful $\I$-cointegral;
    \item every unimodular finite dimensional left coideal subalgebra $\cA$ admitting a faithful $\I$-cointegral is preserved by the antipode squared.
\end{itemize} 
Furthermore we get  \Cref{cor_glS} which can be viewed as a certain   version of   \cite[Proposition 1.6]{LV}; the latter  guaranties,   that every right group-like projection in an algebraic quantum group is two sided. Our result   shows that every right group-like projection in a weakly finite cosemisimple Hopf algebra must be two sided. More generally we proved that every non-degenerate right group-like projection in a cosemisimple Hopf algebra is two sided, see \Cref{rem_LV}.

Using our results and the results of \cite{Chir_Kasp_Szu} we give a list of  all  $\varepsilon$-integrals on left coideal subalgebras of Taft algebras $\cH_{n^2}$  and describe all $g$-cointegrals on them, see \Cref{example_Taft}. 

\section{Preliminaries}\label{Prelim_Sec} All vector spaces considered in this paper are over an arbitrary field $\kk$ if not specified otherwise.
The   vector space dual of $\cV$ will be denoted $\cV'$.  If $\cA$ is an algebra  then $\cA'$ is $\cA$-bimodule, where for $a,b\in\cA$ and $\omega\in\cA'$ we define $\omega\cdot b,a\cdot \omega\in\cA'$ by\[(\omega\cdot b)(a)=(a\cdot \omega)(b) = \omega(ba).\] Let $\cV$ be a left $\cA$-module. We say that an element $v\in\cV$ is cyclic if the map $\cA\ni a\mapsto av\in\cV$ is surjective. If this map is injective then we say that $v$ is separating. Similarly we define cyclic elements for right $\cA$-modules. 
If $\cV$ admits a cyclic element then we say that it is a cyclic module. If the corresponding map $\pi:\cA\to \textrm{End}(\cV)$  is injective then we sat that $\cV$ is faithful. 

A subspace $I\subset \cA$ is said to be a left (right) ideal if $ab\in I$ ($ba\in I$ respectively) for every $a\in\cA$ and $b\in I$. Given a left ideal $I\subset \cA$ we define its right annihilator
\begin{equation}\label{def_lr}r(I) = \{a\in\cA:ba = 0 \textrm{ for all } b\in I\}.\end{equation} The left annihilator $l(I)$  of a  right ideal  is define analogously. Note that if $I$ is a left ideal then  $r(I)$ is a right ideal and $I\subset l(r(I))$. Similarly $I\subset r(l(I))$ for a right ideal $I\subset \cA$. For the needs of our paper we introduce the biannihilator condition.
\begin{definition}\label{biann}
 An algebra $\cA$ is said to satisfy a {\it biannihilator condition} if $l(r(I)) = I$ for every left ideal $I\subset \cA$ and $r(l(I')) = I'$ for every right ideal $I'\subset \cA$. 
 
 A finite-dimensional algebra $\cB$ is quasi-frobenius (or QF) if it is injective as a
module over itself.
\end{definition}
Let us recall that  injectivity of $\cA$-module $\cV$ means that whenever $\cV$ is a submodule of an arbitrary $\cA$ module $\cW$  then it has a complementary submodule in $\cW$. 
It was noted in \cite[Proposition 3.33]{Chir_Kasp_Szu} that QF for $\cA$  implies the  biannihilator condition for $\cA$. 
\begin{definition}\label{def_frob}
Let $\cA$ be a finite dimensional unital algebra  and  $\sigma:\cA\times\cA\to\kk$ a bilinear form on $\cA$. We say that $(\cA,\sigma)$ is a Frobenius algebra if $\sigma$ is  non-degenerate  and satisfies  $\sigma(ab,c) = \sigma(a,bc)$ for all $a,b,c\in\cA$.
\end{definition}
We shall usually write that $\cA$ is a Frobenius algebra having in mind that $\sigma$ entering \Cref{def_frob} is fixed.  
\begin{definition}
Let $\cA$ be a finite dimensional unital algebra. We say that a functional $\omega\in\cA'$ is faithful if  given $a\in\cA$ satisfying    $\omega(ab)=0$ for  all $b\in\cA$ we have $a=0$. 
\end{definition} Let us note that, we can define faithfulness of $\omega\in\cA$ by the condition: if $\omega(ba)=0$ for all $b\in\cA$ then $a=0$. Actually both properties are equivalent,  see e.g. the paragraph following \cite[Definition 1.4]{VD}.  
\begin{theorem}\label{thm_frob_cond}
Let $\cA$ be a finite dimensional unital algebra.
 The following conditions are equivalent:
 \begin{enumerate}
 \item $\cA$ is a Frobenius algebra;
 \item $\cA$ admits a faithful functional;
 \item $\cA$ and $\cA'$ are isomorphic as left (or as right) $\cA$-modules;
 \item $\cA'$ admits a cyclic (or separating) vector.
 \end{enumerate}
\end{theorem}
\begin{remark}\label{rem_f}Let us comment on the proof of \Cref{thm_frob_cond} and incidentally fix some notation. Given a non-degenerate bilinear form $\sigma$ on $\cA$, the functional $\omega_\sigma\in\cA'$ such that $\omega_\sigma(a):=\sigma(\I,a)$ is faithful. Conversely, given a faithful functional $\omega\in\cA$, the bilinear form $\sigma_\omega$ on $\cA$ defined by $\sigma_\omega(a,b) = \omega(ab)$ satisfies the conditions of \Cref{def_frob}. If $\omega\in\cA'$ is a cyclic element then $\cA\ni a \mapsto a\cdot \omega\in\cA'$ is a surjective map of $\cA$-modules, and since $\dim\cA = \dim\cA'$ the map is actually bijective. In particular $\omega$ is faithful.
\end{remark}
For the latter needs let us also observe the following:
\begin{lemma}\label{unital}
Let $\cA$ be a finite dimension algebra. If $\cA$ admits a faithful functional then $\cA$ is unital. 
\end{lemma}
\begin{proof}Suppose that  $\omega\in\cA'$ is faithful. 
Due to the bijectivity of the map $\cA\ni a\to \omega\cdot a\in\cA'$ we can find an element $x\in\cA$ such that $\omega\cdot x = \omega$. One can check that $xa = a$ for all $a\in\cA$, i.e. $x$ is a left unit. Similarly we prove that $\cA$ has a right unit and we are done. 
\end{proof}

We will consider only those  Hopf agebras  $(\mathcal{H},\Delta,S,\varepsilon)$ which have   invertible antipode $S$. Our main reference for Hopf algebras is \cite{Radford_book}. We shall often write that $\mathcal{H}$ is a Hopf algebra, having in mind that comultiplication $\Delta$, coinverse $S$ and counit $\varepsilon$ are fixed.The Sweedler's notation  $\Delta(x) = x_{(1)}\otimes x_{(2)}$ will  used.  A vector subspace $\cV\subset \cH$ is said to be a left coideal if $\Delta(\cV)\subset \cH\otimes \cV$. 
A unital  subalgebra $\cA\subset \cH$ which is a left coideal is said to be a left coideal subalgebra.
Given $\mu,\nu\in\cH'$ and $a\in\cH$ we define \begin{itemize}
    \item $\mu*\nu\in\cH'$ by the formula  $(\mu*\nu)(a)= (\mu\otimes\nu)(\Delta(a))$;
    \item $\mu*a\in\cH$ by the formula  $(\id\otimes\mu)(\Delta(a))$;
    \item $a*\mu\in\cH$ by the formula  $(\mu\otimes\id)(\Delta(a))$.
\end{itemize} Note that $(\cH',*,\varepsilon)$ is a unital algebra. If $\dim\cH<\infty$ then dualizing multiplication $m:\cH\otimes\cH\to\cH$, coinverse $S:\cH\to\cH$ and the unit $\eta:\kk\to \cH$  we get (the dual) Hopf algebra $(\cH',m',S',\eta')$.
If $\cV\subset \cH$ is a left coideal then $\cV'$ is a right $\cH'$-module.  The next lemma will be used later.
 \begin{lemma}\label{unitality} Let $\cH$ be a Hopf algebra 
and $y\in\cH$. If   $y\in\cH$  is a non-zero element satisfying  \begin{equation}\label{P_ident}\Delta(y)(\I\otimes y) = \Delta(y) \end{equation} or $(\I\otimes y)\Delta(y) = \Delta(y)$ then $y=\I$.  
In particular if $\cA\subset \cH$ is a non-zero subalgebra which admits a left or a right unit and satisfies $\Delta(\cA)\subset \cH\otimes\cA$ then $\I_\cH\in\cA$. 
 \end{lemma}
 \begin{proof}Without loss of generality we shall assume that \Cref{P_ident} holds. 
 
 Consider an operator $\ww:\cA\otimes\cA\to\cA\otimes\cA$ such that 
$\ww(a\otimes b) =\Delta(a)(\I\otimes b)$. Its inverse is given by $\ww^{-1}(a\otimes b) = (\id\otimes S)(\Delta(a))(\I\otimes b)$. Noting that  \Cref{P_ident} may be phrased as  $\ww(y\otimes y) = \ww(y\otimes \I)$ we get $y\otimes y = y\otimes\I$ and hence $y=\I$. 

If $\cA$ satisfies the assumptions of the lemma being proven then it has a non-zero element $y\in\cA$ satisfying \Cref{P_ident} or $(\I\otimes y)\Delta(y) = \Delta(y)$ and reasoning as in the previous paragraph we get $y = \I_\cH$.
 \end{proof}
The property defined in the next definition will be useful in the context of coideal subalegebras, see \Cref{rem_skr}. 
\begin{definition}\label{weakfin}
A unital algebra  $\cA$ is said to be  {\it weakly finite} if $xy = \I\implies yx=\I$ for every $x,y\in M_n(\cA)$ and $n\in\mathbb{N}$. 
\end{definition}
\begin{remark}\label{rem_skr}
  Using  \cite[Theorem 6.1]{skr-proj} we see that   if  $\cH$ is Hopf algebra which, us an algebra, is weakly finite and    $\cA\subset \cH$ is a finite dimensional left coideal subalgebra then 
  $\cA$ is Frobenius. Moreover if $\{0\}\neq \cV\subset \cA$ is a left (or right)  ideal in $\cA$ and incidentally it is a left (or right) coideal in $\cH$ then $\cV = \cA$.
\end{remark}

Every element $\Lambda\in\cH$ gives rise to  the smallest left coideal in $\cH$ containing $\Lambda$: \begin{equation}\label{thesmcoi}\cV_\Lambda=\{(\mu\otimes\id)(\Delta(\Lambda)):\mu\in\cH'\}. \end{equation} We shall also consider the smallest right coideal ${}_{\Lambda}{\cV}$ containing $\Lambda$: \[{}_{\Lambda}{\cV}=\{(\id\otimes\mu)(\Delta(\Lambda)):\mu\in\cH'\}.\]

\begin{remark}\label{lem_dim}
For every  $\Lambda\in\cH$ we have $\dim\cV_\Lambda = \dim{}_\Lambda{\cV} $ and it is equal to the rank of the tensor $\Delta(\Lambda)\in\cH\otimes\cH$.   
Noting that $\Delta(\Lambda)\subset {}_{\Lambda}{\cV}\otimes\cV_\Lambda$ we can define a linear bijection  $U:\cV_\Lambda'\to {}_{\Lambda}{\cV}$ by the formula $U(\nu) = (\id\otimes\nu)(\Delta(\Lambda))$.
 \end{remark}
 
The elements $\Lambda,\tilde\Lambda\in\cH$ introduced in  the next definition   turn out to link nicely  with the theory of $\mu$-integrals, see \Cref{f_example}. 
 \begin{definition}\label{def_glel}
A non-zero element $\Lambda\in\cH$ is said to be of integral type if \begin{equation}\label{int_type_def}\Delta(\Lambda)(\I\otimes\Lambda) =(\Lambda\otimes\Lambda)\end{equation} Let $\Lambda$ be an element of integral type. We say that a non-zero element $\tilde\Lambda\in\cH$ is of $\Lambda$-integral type if $\Delta(\tilde\Lambda)(\I\otimes\tilde\Lambda) = \Lambda\otimes\tilde\Lambda$. 
A $\Lambda$-integral element  $\tilde\Lambda$ is  said to be non-degenerate if $\I\in\cV_{\tilde\Lambda}$.
 \end{definition}
 
 \begin{remark}
 Let $\tilde\Lambda\in\cH$ be a non-zero element such that $\Delta(\tilde\Lambda)(\I\otimes\tilde\Lambda) = \Lambda\otimes\tilde\Lambda$ for some element $\Lambda\in\cH$. Then using the techniques of the proof of \cite[Lemma 2.2]{Chir_Kasp_Szu} we can prove that $\Lambda$ is of integral type. 
 
 The related concept of a right group-like projection $P\in\cH$ was recently studied in \cite{Chir_Kasp_Szu}. Let us recall that given $P\in\cH$ satisfying $P^2 = P$, $P\neq 0$ we say that it is
 \begin{itemize}
     \item  a right group-like projection if  
 \[\Delta(P)(\I\otimes P) = (\I\otimes P)\Delta(P) = P\otimes P;\]
  \item  a left  group-like projection if 
 \[\Delta(P)(P\otimes\I) = (P\otimes \I)\Delta(P) = P\otimes P.\]
 \end{itemize} 
  A projection which is incidentally a left and a right group-like projection is called two sided group-like projection. 
  
 If  $\Lambda\in\cH$  is a projection   of integral type   as defined in \Cref{def_glel} then, under certain extra conditions, it is a right group-like projection, c.f. \Cref{onetwo}. 
 \end{remark}
 \begin{remark}\label{rem_til_notil}
Given  an element $\tilde\Lambda$ of $\Lambda$-integral type we define \begin{equation}\label{defal}\cA_{\tilde\Lambda}=\{a\in\cH: \Delta(a)(\I\otimes\tilde\Lambda) = b\otimes\tilde\Lambda \textrm{ for some }b\in\cH\}.\end{equation} Let us note that $\cA_{\tilde\Lambda}$ is a left coideal subalgebra of $\cH$  containing $\cV_{\tilde\Lambda}$. Since  $\ww:\cH\otimes\cH \ni x\otimes y\mapsto \Delta(x)(\I\otimes y) \in\cH\otimes\cH$ is  bijective, the map $\pi:\cA_{\tilde\Lambda}\ni a\mapsto b\in\cH$ (where $a,b\in\cH$ are related as in  \Cref{defal}) is injective and it is easy to see that $\pi(aa') = \pi(a)\pi(a')$ and $\Delta(\pi(a)) = (\id\otimes\pi)(\Delta(a))$. In particular   $\pi(\tilde\Lambda) = \Lambda$,  
$ \pi(\cV_{\tilde\Lambda})=\cV_{\Lambda}$ and $ \pi(\cA_{\tilde\Lambda})=\cA_{\Lambda}
$, and
 $\tilde\Lambda$ is non-degenerate if and only if $\Lambda$ is non-degenerate. 
 
 Applying $(\varepsilon\otimes\id)$ to the identity $\Delta(a)(\I\otimes\tilde\Lambda) = \pi(a)\otimes\tilde\Lambda$ we get  \begin{equation}\label{integr}a\tilde\Lambda = \mu(a)\tilde\Lambda\end{equation} for every $a\in\cA_{\tilde\Lambda}$,
 where $\mu\in\cA_\Lambda'$ is a multiplicative functional given by $\mu(a):=\varepsilon(\pi(a))$.
 
Let us note that the equality  \begin{equation}\label{eq_inv}\Delta(a)(\I\otimes \tilde\Lambda) = \pi(a)\otimes \tilde\Lambda\end{equation} is equivalent to \[(\I\otimes a)\Delta(\tilde\Lambda) = (S(\pi(a))\otimes\I)\Delta(\tilde\Lambda).\] Indeed  \Cref{eq_inv} holds if and only if \begin{equation}\label{eq_inv1}S^{-1}(\tilde\Lambda_{(1)})S^{-1}(a_{(2)})a_{(1)}\otimes a_{(3)}\tilde\Lambda_{(2)} = S^{-1}(\tilde\Lambda_{(1)})\pi(a)\otimes \tilde\Lambda_{(2)} \end{equation} and we conclude the claimed equivalence by noting that $S^{-1}(\tilde\Lambda_{(1)})S^{-1}(a_{(2)})a_{(1)}\otimes a_{(3)}\tilde\Lambda_{(2)} = S^{-1}(\tilde\Lambda_{(1)})\otimes a\tilde\Lambda_{(2)}$ and then applying $(S\otimes\id)$ to both sides of \Cref{eq_inv1}.
\end{remark}
 The results stated in the next theorem  were obtained in \cite{Chir_Kasp_Szu} for right group-like projections; actually the same techniques works for all elements of $\Lambda$-integral type. Note that using the injectivity of $\pi$ entering \Cref{rem_til_notil}, it suffices to prove the next result  for elements of integral type. But for them all arguments provided in the proofs of \cite[Proposition 3.23, Proposition 3.24, Corollary 3.25]{Chir_Kasp_Szu} can be repeated verbatim.  
 \begin{theorem}\label{thm_mod}
 Let $\tilde\Lambda\in\cH$ be  an element of $\Lambda$-integral type, $\cA_{\tilde\Lambda}$ the left coideal subalgebra assigned to $\tilde\Lambda$ and $\pi:\cA_{\tilde\Lambda}\to\cH$ a homomorphism described in \Cref{rem_til_notil}. Then for every $a\in\cA_{\tilde\Lambda}$ we have \begin{equation}\label{eq_pi_til}(\I\otimes a)\Delta(\tilde\Lambda) = (S(\pi(a))\otimes \I)\Delta(\tilde\Lambda).\end{equation} Moreover
 \begin{itemize}
 \item $\cV_{\tilde\Lambda}\subset \cA_{\tilde\Lambda}$ and it is a left-$\cA_{\tilde\Lambda}$ ideal which is faithful when viewed as $\cA_{\tilde\Lambda}$-left module;
 \item  $\dim\cA_{\tilde\Lambda}<\infty$; 
 \item  ${}_{\tilde\Lambda}\cV$ is a right  $\cA_{\tilde\Lambda}$-module via $x\cdot a:=S(\pi(a))x$ for all $a\in\cA_{\tilde\Lambda}$ and $x\in{}_{\tilde\Lambda}\cV$;
 \item the map  $U:\cV_{\tilde\Lambda}'\to{}_{\tilde\Lambda}\cV$ entering the proof of \Cref{lem_dim}  is a right $\cA_{\tilde\Lambda}$-module isomorphism;
 \item $\tilde\Lambda$ is non-degenerate if and only if $\cV_{\tilde\Lambda} = \cA_{\tilde\Lambda}$.
 \end{itemize}
 \end{theorem} 
 Using \Cref{thm_mod} and \Cref{unitality} in the first part of the next corollary and \Cref{rem_skr}   in the the second part  we conclude. 
 \begin{corollary}\label{cor_nondeg}A $\Lambda$-integral element $\tilde\Lambda$ is non-degenerate if and only if $\cV_{\tilde\Lambda}$ is a right or a left unital algebra and the latter is equivalent with $\cV_{\tilde\Lambda} = \cA_{\tilde\Lambda}$. 
 
 Every $\Lambda$-integral  element $\tilde\Lambda$ in a weakly finite Hopf algebra is non-degenerate.
 \end{corollary}

 \begin{problem}
 Is there a Hopf algebra $\cH$ with a degenerate  integral element $\Lambda\in\cH$?
 \end{problem}
 Before formulating the next result let us note  that $\cV_{\tilde\Lambda}\subset \cA_{\tilde\Lambda}$ is a subalgebra; in particular ${}_{\tilde\Lambda}\cV$ can be viewed as $\cV_{\tilde\Lambda}$-module (see the second bullet point of \Cref{thm_mod}). 
 \begin{proposition}\label{cor_frob}
Suppose that ${\tilde\Lambda}$ is a   $\Lambda$-integral element. Then  ${}_{\tilde\Lambda}\cV$ is a cyclic $\cV_{\tilde\Lambda}$-module if and only if  $\I_{\cH}\in \cV_{\tilde\Lambda}$ and $\cV_{\tilde\Lambda}$ is a Frobenius algebra. 
 \end{proposition}
 \begin{proof}
 Suppose that ${}_{\tilde\Lambda}\cV$ is a cyclic $\cV_{\tilde\Lambda}$-module. Due to \Cref{unital} and the fourth bullet point of \Cref{thm_mod} we see that $\cV_{\tilde\Lambda}$ is a unital algebra. Using \Cref{unitality} we see $\I_\cH\in\cV_{\tilde\Lambda}$. Using \Cref{thm_frob_cond} we  conclude that $\cV_{\tilde\Lambda}$ is a Frobenius algebra. Conversely, if  $\cV_{\tilde\Lambda}$ is a Frobenius algebra then ${}_{\tilde\Lambda}\cV$ is a cyclic due to \Cref{thm_frob_cond} and the fourth bullet point of \Cref{thm_mod}.
 \end{proof}
\section{Generalized integrals for left coideal subalgebras.}\label{sec_muint}
\begin{definition}
Let $\cA$ be an algebra and $\mu\in\cA'$ a non-zero multiplicative functional on $\cA$. We define 
\begin{align*}L^\cA_\mu&=\{\Lambda\in\cA:a\Lambda= \mu(a)\Lambda\textrm{ for all } a\in\cA\},\\
R^\cA_\mu&=\{\Lambda\in\cA:\Lambda a= \mu(a)\Lambda\textrm{ for all } a\in\cA\}.\end{align*}
 The elements of   $L^\cA_\mu$ ($R^\cA_\mu$) are called left (resp. right) $\mu$-integrals. 
\end{definition}
The proof of the next lemma is the direct consequence of \Cref{def_lr} and \Cref{biann}.
\begin{lemma}\label{bilem}
Let $\cA$ be an algebra and $\mu\in\cA'$ a non-zero multiplicative functional on $\cA$. Then $L^\cA_\mu = r(\ker\mu)$ and $R^\cA_\mu = l(\ker\mu)$. In particular if $\cA$ satisfies the biannihilator condition then $\dim L^\cA_\mu >0$ and $\dim R^\cA_\mu >0$. More specifically the latter holds if $\cA$ satisfies the QF property. 
\end{lemma}
\begin{example}\label{f_example}
Let $\cH$ be a Hopf algebra and let  $\tilde\Lambda\in\cH$, $\cA_{\tilde\Lambda}$ and  $\mu\in\cA_{\tilde\Lambda}'$ be the $\Lambda$-integral element, left coideal subalgebra and multiplicative functional respectively, as defined in \Cref{rem_til_notil}. Equation \Cref{integr} means that $\tilde\Lambda\in L_\mu^{\cA_{\tilde\Lambda}}$. 
\end{example}
\begin{remark}\label{rem_semi1}
Note that  for every $a\in\cA$ and $\Lambda\in L_\mu^\cA$  we   have $a\Lambda\in L_\mu^\cA$. Moreover,  given $b\in\cA$ we have $a\Lambda b = \mu(a)\Lambda b$ for all  $a\in\cA$ and we see that $L_\mu^\cA$ is a two sided ideal in $\cA$. Similarly we check that $R_\mu^\cA$ is a two sided ideal.

Suppose that $\cA$ is semisimple. Then since $\ker\mu\subset \cA$ is a two sided ideal, there is a unique central projection $Q\in\cA$ such that $Q\cA = \ker\mu$. Putting $P = \I-Q$ we have $P\in L_\mu^\cA\cap R_\mu^\cA$. It is easy to check that $L^\cA_\mu = R^\cA_\mu = \kk P$. 
\end{remark}
\begin{theorem}\label{thm_dim}
Let $(\cA,\sigma)$ be a Frobenius algebra and $\mu\in\cA'$ a non-zero multiplicative functional. Then $\dim L^\cA_\mu  = \dim R^\cA_\mu = 1$. Moreover every $1$-dimensional one sided  sided ideal $L \subset \cA$ is two sided and it is of the form $L_\mu^\cA$ for a unique  multiplicative functional $\mu$ on $\cA$. 
\end{theorem}
\begin{proof}
In the proof we shall use the notation introduced in \Cref{rem_f}.  

Note that  $\Lambda\in L_\mu^\cA$
if and only if $\Lambda\cdot\omega_\sigma = \omega_\sigma(\Lambda)\mu$. Indeed, applying $\omega_\sigma$ to  $a\Lambda = \mu(a)\Lambda$ we get the left implication. Conversely, if $\Lambda\cdot\omega_\sigma = \omega_\sigma(\Lambda)\mu$ then for all $b\in\cA$ we have
\begin{align*}\omega_\sigma(ba\Lambda)&=(\Lambda\cdot\omega_\sigma)(ba)\\&=
\omega_\sigma(\Lambda)\mu(ab)\\&=
\omega_\sigma(\Lambda)\mu(\mu(a)b)\\
&=\omega_\sigma(b\mu(a)\Lambda)
\end{align*} and hence $a\Lambda = \mu(a)\Lambda$. 
Let $\iota:\cA\to\cA'$ be the bijection given by $\cA\ni x\mapsto x\cdot\omega_\sigma\in\cA'$. The above reasoning shows that $\iota(L_\mu^\cA) = \kk\mu$ and hence $\dim L_\mu^\cA =1$. Similarly we prove that $\dim R_\mu^\cA = 1$.

If  $L\subset \cA$ is  a $1$-dimensional, say left  ideal, then there exists a multiplicative functional  $\mu\in\cA\to\kk$ such that $a\Lambda = \mu(a)\Lambda$ for every $a\in\cA$ and $\Lambda\in L$ and hence $L = L_\mu^\cA$.
\end{proof}
Using  \cite[Theorem 6.1]{skr-proj} (see \Cref{rem_skr}) we get:
\begin{corollary}\label{weak_f}
Let $\cH$ be a weakly finite Hopf algebra and $\cA$  a finite dimensional left coideal subalgebra of $\cH$ and let $\mu\in\cA'$ be a  non-zero multipilicative functional. Then $\dim L^\cA_\mu = 1 = \dim R^\cA_\mu$. 
\end{corollary}

\begin{theorem}\label{thm_fd}
Let $\cA\subset \cH$ be a left  coideal subalgebra and $\mu\in\cA'$  a multiplicative functional on $\cA$. If $\dim L^\cA_\mu>0$ or $\dim R^\cA_\mu>0$ then $\dim\cA<\infty$. 
\end{theorem}
\begin{proof} We proceed assuming that $\dim L^\cA_\mu>0$. 
Let $\tilde\Lambda\in L^\cA_\mu$ be a non-zero element and $\cA_{\tilde\Lambda}$ the left coideal subalgebra introduced in \Cref{rem_til_notil}. Noting that  $\cA\subset \cA_{\tilde\Lambda}$ and using the second bullet point of \Cref{thm_mod}  we conclude that $\dim\cA<\infty$.  
\end{proof}

 \begin{problem}
 Is there a finite dimensional left coideal subalgebra $\cA\subset \cH$ equipped with a non-zero multiplicative functional  such that $L^\cA_\mu = 0$  or $R^\cA_\mu = 0$?
 
  Is there a finite dimensional left coideal subalgebra $\cA\subset \cH$ equipped with a non-zero multiplicative functional  such that $\dim L^\cA_\mu >1$  or $\dim R^\cA_\mu >1$?
  
  Can we have $\dim L^\cA_\mu  \neq \dim R^\cA_\mu  $?
 \end{problem}
\begin{definition}\label{non-deg-L}
  Let $\cA$ be a left coideal subalgebra of Hopf algebra $\cH$ and $\mu\in\cA'$ a non-zero multiplicative functional. We say that $\tilde\Lambda\in L_\mu^\cA$ is non-degenerate if $\cA = \cV_{\tilde\Lambda}$. 
\end{definition}
\begin{remark}\label{lamint}Let $\cA\in\cH$ be a finite dimensional left coideal subalgebra.
If  $0\neq \tilde\Lambda\in L^\cA_\mu$ then 
\[\Delta(\tilde\Lambda)(\I\otimes\tilde\Lambda) = \Lambda\otimes\tilde\Lambda\] where $\Lambda =  (\id\otimes\mu)\circ\Delta(\tilde\Lambda)$. In particular $\tilde\Lambda$ is a $\Lambda$-integral element. In what follows   the map $(\id\otimes\mu)\circ\Delta :\cA\to \cH$ will be denoted $\pi_\mu$. Reasoning as in \Cref{rem_til_notil} we see that $\pi_\mu$ is injective if $L^\cA_\mu\neq \{0\}$. 
\end{remark}In the next proposition we shall use the notation of \Cref{rem_til_notil} and \Cref{lamint}.
\begin{proposition}\label{non_deg}
 Let $\cA$ be a left coideal subalgebra of a Hopf algebra $\cH$ and $\mu\in\cA'$ a non-zero multiplicative functional.  The following conditions are equivalent:
 \begin{itemize}
     \item $\tilde\Lambda\in L_\mu^\cA$ is non-degenerate in the sense of \Cref{non-deg-L};
     \item $\tilde\Lambda$ is a non-degenerate $\Lambda$-integral element in the sense of \Cref{def_glel};
     \item we have $\cV_{\tilde\Lambda} = \cA_{\tilde\Lambda}$.
 \end{itemize}
 In particular if $\cH$ is weakly finite then every non-zero $\Lambda\in L_\mu^\cA$ is non-degenerate.
\end{proposition}
\begin{proof}
Note that $\cV_{\tilde\Lambda}\subset\cA\subset \cA_{\tilde\Lambda}$   and since  $\cV_{\tilde\Lambda}$ is a left ideal in $\cA_{\tilde\Lambda}$ (c.f. \Cref{thm_mod}),  all bullet points are equivalent with the condition $\I_\cH\in\cV_{\tilde\Lambda}$. The last claim of the proposition  follows from \Cref{cor_nondeg}. 
\end{proof}
In the next proposition we describe the relations between biannihilator condition from \Cref{biann}, non-degeneracy condition and one-dimensionality of $L_\mu^\cA$.
\begin{proposition}\label{biann_non} Let $\cA\subset \cH$ be a left coideal subalgebra satisfying the biannihilator condition. Then 
\begin{itemize}
    \item if $\tilde\Lambda\in L_\mu^\cA$ is non-degenerate and $\mu(\tilde\Lambda) \neq 0$ then $L_\mu^\cA = R_\mu^\cA = \kk\tilde\Lambda$; 
    \item if $L_\mu^\cA = \kk\hat\Lambda$  for some $0\neq \hat\Lambda\in L_\mu^\cA $  then $\hat\Lambda$ is non-degenerate. In particular if $\cA$ is Frobenius then $\hat\Lambda$ is non-degenerate. 
\end{itemize}
\end{proposition}
\begin{proof}
In order to prove the first bullet point of the proposition we first use non-degeneracy of $\tilde\Lambda$ to conclude that  given  $\tilde\Lambda'\in R^\cA_\mu$, there exists $\omega\in\cH'$ such that $\tilde\Lambda' = (\omega\otimes\id)(\Delta(\tilde\Lambda))$. In particular we have
\begin{equation}\label{eqti}\tilde\Lambda'\tilde\Lambda =\tilde\Lambda'\mu(\tilde\Lambda) = \omega(\Lambda)\tilde\Lambda \end{equation} where $\Lambda = \pi_\mu(\tilde\Lambda)$, see \Cref{lamint}. Dividing  \Cref{eqti} by  $\mu(\tilde\Lambda)$ we conlcude that $\tilde\Lambda'$ is a multiple of $\tilde\Lambda$, and hence  $\dim R_\mu^\cA\leq 1$. Using \Cref{bilem} we get   $R_\mu^\cA =\kk\tilde\Lambda$. In particular  $\kk\tilde\Lambda$ is a two sided ideal and we have   $\ker\mu = l(\kk\tilde\Lambda) = r(\kk\tilde\Lambda)$. The biannihilator condition together with \Cref{bilem} imply that $L_\mu^\cA = \kk\tilde\Lambda$.

In order to prove the claim entering the second bullet point of the proposition let us note that since $\kk\hat\Lambda$ is a two sided ideal, there exists a non-zero multiplicative functional $\nu\in\cA'$ such that $\hat\Lambda\in R_\nu^\cA$. Thus  for every $a\in\cA$ we have $(\I\otimes \hat\Lambda)\Delta(a)= \pi_\nu(a)\otimes\hat\Lambda$. Equivalently $\Delta(\hat\Lambda)(S^{-1}(\pi_\nu(a))\otimes\I) = \Delta(\hat\Lambda)(\I\otimes a)$ (c.f. \Cref{rem_til_notil}) and we conclude that $\cV_{\hat\Lambda}$ is right ideal in $\cA$. Using \Cref{thm_mod} we see that $\cV_{\hat\Lambda}$  is a two sided faithful ideal in $\cA$ thus 
\[l(\cV_{\hat\Lambda}) =\{a\in\cA: av = 0 \textrm{ for all } v\in\cV_\Lambda\} = \{0\}.\]
On the other hand, if $\cV_{\hat\Lambda}\subsetneq\cA$ and incidentally $\cA$ satisfies the biannihilator condition then $l(\cV_{\hat\Lambda})\neq \{0\}$ - contradiction. Hence    $\cV_{\hat\Lambda}=\cA$, which is one of the equivalent form of non-degeneracy of $\hat\Lambda$, c.f. \Cref{non_deg}.

Note that if $\cA$ is Frobienius then it is QF and thus it satisfies the bianihilator condition (see \Cref{biann} and the following  paragraph). We conclude the second part of the second bullet point using the first part and \Cref{thm_dim}. 
\end{proof}
The next result should be compared with  \Cref{biann_non}. 
\begin{proposition}\label{rem_twosided}
Let $\cA\subset\cH$ be a finite dimensional left coideal subalgebra and let   ${\tilde\Lambda}\in L^\cA_\mu$, $\mu(\tilde\Lambda)\neq 0$. 
If $\dim L_\mu^\cA= 1$ then $L_\mu^\cA= R_\mu^\cA = \kk\tilde\Lambda$.
\end{proposition}
\begin{proof}
Since   $L^\cA_\mu$ is a two sided 1-dimensional ideal,  there  exists a non-zero multiplicative functional   $\nu\in\cA'$ such that  ${\tilde\Lambda} a = \nu(a){\tilde\Lambda}$ for every $a\in \cA$, i.e. $\tilde\Lambda\in R_\nu^\cA$. 
Moreover  \[\mu({\tilde\Lambda} a) =\mu(a) \mu({\tilde\Lambda}) =\nu(a) \mu({\tilde\Lambda}).\] Thus we get  $\mu = \nu$ and $L^\cA_\mu\subset  R^\cA_\mu$. 
Furthermore if  $\tilde\Lambda'\in R^\cA_\mu$ then $\tilde\Lambda'\mu(\tilde\Lambda)=\tilde\Lambda'\tilde\Lambda = \tilde\Lambda \mu(\tilde\Lambda')$ and dividing both sides by $\mu(\tilde\Lambda)$ we conclude that $ R^\cA_\mu = \kk\tilde\Lambda$.
\end{proof}
\begin{definition}\label{unimod}
Let $\cA$ be a left coideal subalgebra of a Hopf algebra $\cH$. We say that $\cA$ is unimodular if $L_\varepsilon^\cA = R_\varepsilon^\cA$.
\end{definition}
  \Cref{biann_non}, \Cref{rem_twosided}   together with \cite[Theorem 6.1]{skr-proj} yield the next corollary.
\begin{corollary}\label{onetwo}
If $\cA$ is a left coideal  subalgebra of $\cH$, $\dim(L^\cA_\varepsilon) = 1$ and $\varepsilon(L_\varepsilon^\cA)\neq\{ 0\}$ then $\cA$ is unimodular.

If   $P\in\cH$ is a  projection  satisfying $\Delta(P)(\I\otimes P) = P\otimes P$ such that $\cV_P$ is a unital algebra satisfying  the biannihilator condition
 then $P$ is a right group-like projection.

If   $P\in\cH$ is a non-zero projection  satisfying $\Delta(P)(\I\otimes P) = P\otimes P$ and such that $\cA_P$ is Frobenius, then $P$ is a right group-like projection. 

The latter  holds if $\cH$ is weakly finite. 
\end{corollary}

Given a pair of  elements $\Lambda_1,\Lambda_2\in\cH$ of integral type we say that they are equivalent if there exists $t\in\kk$ such that $\Lambda_1 = t\Lambda_2$. The equivalence class containing $\Lambda$ will be denoted by $[\Lambda]$.

The results of this section together with  \cite{skr-fin}   yield the following theorem. 
\begin{theorem}\label{11corr}
Let $\cH$ be a weakly finite Hopf algebra. Then every element $\Lambda$ of integral type is non-degenerate. Moreover the map $[\Lambda]\mapsto \cA_\Lambda$  establishes a 1-1  correspondence between finite dimensional left coideal subalgebras of $\cH$ and classes of integral type elements in $\cH$. In particular their number is finite  if $\cH$ is  a simple or cosemisimple finite dimensional Hopf algebra.
\end{theorem}

\section{Generalized cointegrals on left coideal subalgebras with integral}\label{Sec_gcoint}
Let us recall that a non-zero element $g$ of a coalgebra $\cC$ is said to be a group-like element if $\Delta(g) = g\otimes g$.
\begin{definition}\label{def_gcoint}
Let $\cV\subset \cC$ be a left coideal in a coalgebra $\cC$   and $g\in\cC$ a non-zero group-like element. We say that a functional $\phi\in\cV'$  is a $g$-cointegral on $\cV$ if $(\id\otimes\phi)(\Delta(a)) = \phi(a)g$ for all $a\in\cV$.  The subspace of $g$-cointegrals on $\cV$ will be denoted by $L^g_\cV$. 
\end{definition}

 Given an element $\Lambda$ in a coalgebra $\cC$ we can assign it with the smallest left (right) coideal  containing $\Lambda$    denoted $\cV_\Lambda$ (${}_\Lambda\cV$ respectively) and we have $\dim \cV_\Lambda = \dim{}_\Lambda\cV$, c.f. \Cref{lem_dim} and the paragraph before.  In the next proposition we shall use the bijection $U:\cV_\Lambda'\to {}_\Lambda\cV$ where \begin{equation}\label{def_U}U(\nu) = (\id\otimes\nu)(\Delta(\Lambda)).\end{equation}
\begin{proposition}\label{dim_cgint}Let $\cC$ be a  coalgebra, $g\in\cC$ a group-like element and $\Lambda\in \cC$. 
If $g\in {}_\Lambda\cV$ then $\dim L^g_{\cV_\Lambda}= 1$  and $\dim L^g_{\cV_\Lambda}=0$ otherwise. Moreover $\phi\in L^g_{\cV_\Lambda}$ is non-zero if and only if $\phi(\Lambda)\neq 0$.
\end{proposition}
\begin{proof}
It is easy to check that $\phi\in\cV_\Lambda'$ is a $g$-cointegral if and only if 
\begin{equation}\label{gint_cond}U(\phi) = g\phi(\Lambda)\end{equation} where $U:\cV_\Lambda'\to {}_\Lambda\cV$ is given by \Cref{def_U}. In particular $\phi\neq 0 $ if and only if $\phi(\Lambda)\neq 0$ and    $\dim L^g_{\cV_\Lambda}\leq 1$. Clearly if $\dim L^g_{\cV_\Lambda}=1$ then  $g\in {}_\Lambda\cV$. Conversely if  $g\in {}_\Lambda\cV$ then  defining   $\phi = U^{-1}(g)\in\cV_\Lambda'$ we get a  non-zero  $g$-cointegral.
\end{proof}
Using \Cref{dim_cgint}, and \Cref{cor_frob} we get  the next result. 
\begin{theorem}\label{thm_gcoin}Let $\tilde\Lambda\in\cH$ be an element  of $\Lambda$-integral type and $g\in\cH$   a group-like element. Then  $\dim L^g_{\cV_{\tilde\Lambda}} = 1$ if and only if $g\in{}_{\tilde\Lambda}\cV$ and  $\dim L^g_{\cV_{\tilde\Lambda}} = 0$ otherwise. Moreover if $\phi\in L^g_{\cV_{\tilde\Lambda}}$ and $\phi\neq 0$ then ${\tilde\Lambda}$ is non-degenerate and  $\phi$ is faithful. Conversely if $\cA_{\tilde\Lambda}$ admits a faithful $g$-cointegral then ${\tilde\Lambda}$ is non-degenerate. 
 \end{theorem}
\begin{proof}
Applying \Cref{dim_cgint} to $\cV_{\tilde\Lambda}$ we conclude that  $\dim L^g_{\cV_{\tilde\Lambda}} = 1$ if and only if $g\in{}_{\tilde\Lambda}\cV$ and  $\dim L^g_{\cV_{\tilde\Lambda}} = 0$ otherwise. If $\phi\in L^g_{\cV_{\tilde\Lambda}}$ is non-zero then $\I_\cH\in\cV_{\tilde\Lambda}$ (i.e. $\tilde\Lambda$ is non-degenerate) and $\phi$ is faithful, c.f. \Cref{cor_frob}. 

If  $\cA_{\tilde\Lambda}$ admits a faithful $g$-cointegral then  $\tilde\Lambda$ is non-degenerate due to the second bullet point of \Cref{biann_non}.
\end{proof}

For the sake of simplicity of  presentation the next results  will be formulated for the elements of integral types. Their generalizations  to elements of $\Lambda$-integral types can easily be written  on the basis of  \Cref{rem_til_notil}: one must replace $\Lambda$ with $\tilde\Lambda$ and $S$ with $S\circ\pi$ where $\pi$ is the  homomorphism entering \Cref{rem_til_notil}.



\begin{theorem}\label{thm_two_sided}
  Let $\Lambda\in\cH$ be an  element of integral type.
Then $\cA_\Lambda$ admits a faithful  $g$-cointegral if and only if   ${}_\Lambda\cV = S(\cA_\Lambda)g$.
In particular  if $\cA_\Lambda$ admits a faithful non zero $g$-cointegral  then
\begin{itemize}
    \item $(\I\otimes \Lambda)\Delta(\Lambda) = S(\Lambda)g\otimes \Lambda$;
    \item $L_\varepsilon^\cA=R_\varepsilon^\cA$  if and only if  $\Lambda = S(\Lambda)g$.
\end{itemize}  
\end{theorem}
\begin{proof} 
If   ${}_\Lambda\cV  = S(\cA_\Lambda)g$ then  $g\in{}_\Lambda\cV$ and  using  \Cref{thm_gcoin} we get  $\cV_\Lambda = \cA_\Lambda$. In particular  $\cV_\Lambda$ is a unital algebra admitting a faithful $g$-cointegral. 
Conversely, if $\cA_\Lambda$ admits a faithful $g$-cointegral then using \Cref{thm_gcoin} we get $\cA_\Lambda = \cV_\Lambda$. In particular $g$ is a cyclic element of ${}_\Lambda\cV$, i.e. ${}_\Lambda\cV = S(\cV_\Lambda)g$.

In order to prove the first and the second bullet point note that under our assumptions  $\cA_\Lambda$  is a Frobenius algebra. In particular  $\kk\Lambda$ is a two sided ideal  in $\cA_\Lambda$ and we have \[ (\I\otimes\Lambda)\Delta(\Lambda) = \tilde\Lambda\otimes \Lambda\] for some $\tilde\Lambda\in{}_\Lambda\cV$. Using \Cref{eq_pi_til}  with $x=\Lambda$ and remembering that $\pi=\id$ we get
\begin{equation}\label{leftin}(S(\Lambda)\otimes\I)\Delta(\Lambda) = \tilde\Lambda\otimes \Lambda.\end{equation}
If  $\phi\in\cA_\Lambda'$ is a non-zero $g$-cointegral then applying $(\id\otimes\phi)$ to \Cref{leftin} we get  $S(\Lambda)g\phi(\Lambda) = \tilde\Lambda\phi(\Lambda)$. Since $\cA_\Lambda = \cV_\Lambda$ (c.f. \Cref{thm_gcoin}) we conclude using \Cref{dim_cgint} that $\phi(\Lambda)\neq 0$ and we get $S(\Lambda)g = \tilde\Lambda$.
\end{proof}
\begin{remark}\label{unimod1}
If $\Lambda\in\cH$ is an element of integral type such that $S(\Lambda) = \Lambda$ then it is easy to check that ${_\Lambda}\cV = S(\cV_\Lambda)$. Thus if $\Lambda$ is non-degenerate then $\cV_\Lambda$ admits faithful  $\I$-cointegral. In particular using \Cref{thm_two_sided} we conclude that $\cV_\Lambda$ is unimodular. In the next result we further link the  unimodularity property for coideals  with the existence of  non-trivial $\I$-cointegrals on them.
\end{remark}
\begin{theorem}\label{uni_cor}
Let $\Lambda\in\cH$ be an element of integral type. Then:
\begin{itemize}
    \item if $\cA_\Lambda$ is unimodular and it admits a faithful $\I$-cointegral then $S(\Lambda) = \Lambda$ and  $S^2(\cA_\Lambda) = \cA_\Lambda$;
    \item if $\Lambda$ is a projection and $\cA_\Lambda$ admits a faithful $\I$-cointegral then $\cA_\Lambda$ is unimodular, we have  $S^2(\cA_\Lambda) = \cA_\Lambda$, $\Lambda$ is two sided group-like projection and $\cA_\Lambda$ is semisimple; 
\item if $\cA$ is a semisimple left coideal subalgebra of $\cH$  preserved by $S^2$ then $\cA$ admits a unique (up multiplicative constant)  faithful $\I$-cointegral.
\end{itemize}
\end{theorem}
\begin{proof}
If $\cA_\Lambda$ is unimodular and it admits faithful $\I$-cointegral then $S(\Lambda) = \Lambda$ by \Cref{thm_two_sided}. Moreover  $S^2(\cA_\Lambda) = \cA_{S^2(\Lambda)} = \cA_{\Lambda}$ and we get the first bullet point of out theorem.

In order to get the second bullet point assume that  $\Lambda^2=\Lambda$ . Since $\cA_\Lambda$ is Frobenius we conclude that it is unimodular by \Cref{onetwo}. In particular $(\I\otimes\Lambda)\Delta(\Lambda) = \Lambda\otimes\Lambda = \Delta(\Lambda)(\I\otimes\Lambda)$ and $S(\Lambda) =\Lambda$ by the first bullet point, i.e. $\Lambda$ is a two sided group-like projection. Using \cite[Theorem 3.22.]{Chir_Kasp_Szu}  we see that $\cA_\Lambda$ is semisimple.

In order to get the third bullet point let us assume that $\cA$ is a semisimple left coideal subalgebra of $\cH$  preserved by $S^2$. There exists a unique two sided group-like projection $\Lambda\in\cH$ such that $\cA= \cA_\Lambda=\cV_\Lambda$ (see \cite[Theorem 3.27]{Chir_Kasp_Szu} and \cite[Theorem 3.26]{Chir_Kasp_Szu}). The equality   $S(\Lambda) = \Lambda$ implies then that $S(\cV_\Lambda) ={}_\Lambda\cV$ and we conclude that $\dim L^\I_{\cA_\Lambda} = 1$ using \Cref{thm_gcoin}.
\end{proof}

Every right group-like projection in an algebraic quantum group is automatically two sided, see   \cite[Proposition 1.6]{LV}. In what follows we prove the counterpart of this result for cosemisimple weakly finite Hopf algebras.
\begin{theorem}\label{cor_glS}
Let $\cH$ be a weakly finite cosemisimple Hopf algebra and $\Lambda$ an integral type element.  Then $S(\Lambda) = \Lambda$ if and only if $(\I\otimes\Lambda) \Delta(\Lambda)=\Lambda\otimes\Lambda$. In particular every finite dimensional unimodular left coideal subalgebra $\cA\subset \cH$   is preserved by $S^2$ and  every    integral type projection  is a two sided group-like projection.
\end{theorem}

\begin{proof}
Let us recall that every integral type element in a weakly finite Hopf algebra is non-degenerate and   every coideal subalgebra $\cA\subset \cH$ is of the form $\cA = \cV_\Lambda$, c.f. \Cref{11corr}.  

If $S(\Lambda) = \Lambda$ then  since  $\Lambda$ is non-degenerate we can conclude that  $(\I\otimes\Lambda)\Delta(\Lambda) = \Lambda\otimes\Lambda$ using \Cref{unimod1}. Conversely if $(\I\otimes\Lambda)\Delta(\Lambda) = \Lambda\otimes\Lambda$ then $\cV_\Lambda$ is unimodular and it admits a faithful $\I$-cointegral (the restriction of the unital $\I$-cointegral on $\cH)$. Using \Cref{uni_cor} we see that $S(\Lambda) = \Lambda$. 

If $\cA$ is a unimodular finite dimensional left coideal subalgebra of $\cH$ then using the previous paragraph we get   $\Lambda\in\cH$ such that $\cA = \cV_\Lambda$ and $S(\Lambda) = \Lambda$. In particular $S^2(\cA) = \cA$. 

Finally if $\Lambda^2 = \Lambda$ then using \Cref{onetwo} we see that $(\I\otimes\Lambda)\Delta(\Lambda) = \Lambda\otimes\Lambda$ and hence $S(\Lambda) = \Lambda$ (by the first part of the theorem being proven), i.e. $\Lambda$ is a two-sided group-like projection.
\end{proof}
\begin{remark}\label{rem_LV}
Let $\cH$ be  finite dimensional  cosemisimple  Hopf algebra. It still unknown if $S^2=\id$ but our previous result shows at least  that unimodular left coideal subalgebras of $\cH$ must be preserved by $S^2$ in this case. Unfortunately we were not able to drop unimodularity assumption in this result.  

Let us note that the method of the proof of \Cref{cor_glS}  enables us also to prove that every non-degenerate integral type projection $\Lambda$ in a cosemisimple Hopf algebra $\cH$ is a two sided projection. Indeed, since the restriction of (unital)  $\I$-cointegral from $\cH$  to (unital) algebra $\cV_\Lambda$ is non-zero, $\cV_\Lambda$ itself admits a non-zero $\I$-cointegral $\phi\in L^\I_{\cV_\Lambda}$. Using \Cref{thm_gcoin} we see that $\phi$ is faithful and thus $\cV_\Lambda$ is unimodular (see \Cref{onetwo}). Using \Cref{uni_cor} we conclude that $S(\Lambda) = \Lambda$ and we are done. 
\end{remark}
The next result links   semisimplicity of $\cV_\Lambda$ with the fact that it is Frobenius and it can be viewed as a generalization of \cite[Theorem 3.22]{Chir_Kasp_Szu}. Let us recall that $\cV_\Lambda$ is Frobenius if and only if ${}_\Lambda\cV$ is a cyclic $\cV_\Lambda$-module (c.f. \Cref{cor_frob}). 
\begin{theorem}\label{thm_semi}
Let $P\in\cH$ be a right  group-like projection and assume that $\cV_P$ is a Frobenius algebra with a cyclic (hence separating) element $y\in{}_P\cV$. Consider the map $\iota_y:{}_P\cV\to \cV_P$  given by $\iota_y(S(a)y) = a$ and define  $z_y = P_{(2)}\iota_y(P_{(1)})\in\cV_\Lambda$. If $z_y$ is  an invertible element of $\cV_P$ then every right $\cV_P$-module is completely reducible. In particular $\cV_P$ is semisimple. 
\end{theorem}
\begin{proof}

Note that if $a\in\cV_P$ and $x = S(a)y$ then $\iota_y(S(b)x) =\iota_y(S(b)S(a)y) = ab=\iota_y(x)b$ for all $b\in\cV_P$.
In particular, using the identity $(\I\otimes b)\Delta(P) = (S(b)\otimes\I)\Delta(P)$, we get \begin{align*}bz_y = bP_{(2)}\iota_y(P_{(1)}) &= P_{(2)}\iota_y(S(b)P_{(1)})\\& = P_{(2)}\iota_y(P_{(1)})b = z_yb\end{align*} for all $b\in\cV_P$ and hence $z_y$ is a central (invertible) element in $\cV_P$.

Suppose that $M$ is a right $\cV_P$-module and $N\subset M$ is a submodule and let $q:M\to M$ be a linear  projection onto $N$. Let us define $Q_0:M\to M$ by the formula
\[Q_0(x) = q(xP_{(2)})\iota_y(P_{(1)}).\] Using $(\I\otimes b)\Delta(P) = (S(b)\otimes\I)\Delta(P)$ again we get
\begin{align*}Q_0(xb)&=q(xbP_{(2)})\iota_y(P_{(1)})\\ &= q(xP_{(2)})\iota_y(S(b)P_{(1)})\\&= Q_0(x)b\end{align*} and thus $Q_0$ is a $\cV_P$-module  map. Moreover if $x\in N$ then $Q_0(x) = xP_{(2)}\iota_y(P_{(1)})=xz_y$. Hence if we put  $Q(x) = Q_0(x)z_y^{-1}$ then $Q:M\to M$ is still a $\cV_{P}$-module map  (note that we use here the centrality of $z_y^{-1}$) which is a  projection onto $N$. Taking $N'=\ker Q$ we get decomposition  $M = N\oplus N'$  and we are done. 
\end{proof}
\begin{remark} 
Assuming that $y\in{}_P\cV$ entering the assumptions of  \Cref{thm_semi} is an invertible element of $\cH$ we have  $z = \textrm{Ad}_P(S(y^{-1}))$ where we define $\textrm{Ad}_P(a) = P_{(2)}aS^{-1}(P_{(1)})$ for all $a\in\cH$. In particular if $\cV_P$ admits a $g$-cointegral then we can take $y=g$ (c.f. \Cref{thm_gcoin}) and in this case $z_g=\textrm{Ad}_P(g)$. Unfortunately we were not able to prove that $z_g$ is necessarily invertible. In \Cref{example_Taft}   we compute  those elements for semisimple left coideal subalgebras of Taft algebras showing   that they are invertible in this case. 
\end{remark}
  
Let us finish this section with a recapitulation of  the results of this paper  for  finite dimensional Hopf algebras.
\begin{theorem}
Let $\cH$ be a finite dimensional Hopf algebra $\cA\subset\cH$ a left coideal subalgebra $\mu\in\cA'$ a non-zero multiplicative functional and $g\in\cH$ a group-like element.
\begin{itemize}
    \item [(a)] The ideal $L_\mu^\cA$ of left $\mu$-integrals is one dimensional.
    \item[(b)] Let $\tilde\Lambda\in L_\mu^\cA$, $\tilde\Lambda\ne 0$ and $\phi\in L^g_\cA$. Then $\phi\neq 0$ if and only if $\phi(\tilde\Lambda)\neq 0$. 
    \item[(c)] Suppose that $\tilde\Lambda\in L_\mu^\cA$, $\tilde\Lambda\neq 0$ and $\phi\in L_g^\cA$ is such that $\phi(\tilde\Lambda) = 1$. Then 
    \begin{itemize}
        \item $\tilde\Lambda*(\phi\cdot a)=S(\pi(a))g$ for all $a\in\cA$  where $\pi:\cA\to \cH$ is the homomorphism described in \Cref{rem_til_notil}.
        \item $\tilde\Lambda*\phi = g$ and $\tilde\Lambda\cdot\phi = \mu$.
    \end{itemize}
    \item[(d)] For every group-like element  $g\in\cH$ every $\nu\in\cH'$ and every $\phi\in L_g^\cA$ we have $\nu*\phi = \nu(g)\phi$. Moreover $\dim L_g^\cA \in\{0,1\}$ and $\dim L_g^\cA = 1$ if there exists $\phi\in\cA'$ such that $\tilde\Lambda*\phi = g$. In this case $\phi$ is a faithful functional on $\cA$. Conversely if $L\subset \cA'$ is $1$-dimensional $\cH'$-submodule of $\cA'$ then there exists $g\in\cH$ such that $L = L_g^\cA$. 
\end{itemize}
\end{theorem}
\section{Integrals and cointegrals for left coideal subalgebras of Taft Hopf algebras}\label{example_Taft}
Following \cite{tft}, we have
\begin{definition}\label{def.tft}
  Let $n$ be a positive integer and $\omega\in\CC$ a primitive $n^{th}$ root of unity. The {\it Taft Hopf algebra} $\cH_{n^2}$ is the algebra $\CC[x]/(x^n)\rtimes \CC[G]$, where $G=\mathbb{Z}/n\mathbb{Z}$ is generated by $g$ and its action on the truncated polynomial ring $\CC[x]/(x^n)$ is by
  \begin{equation*}
    gxg^{-1} = \omega x. 
  \end{equation*}
  The coalgebra structure is given by
  \begin{equation*}
    \Delta(g) = g\otimes g,\quad \Delta(x)=x\otimes 1 + g\otimes x. 
  \end{equation*}
\end{definition}
In what follows  we give the formulas for  $\varepsilon$-integrals assigned to left coideal subalgebras of Taft Hopf algebra  $\cH_{n^2}$. See \cite[Section 4]{Chir_Kasp_Szu} for the description of all left coideals subalgebras of $\cH_{n^2}$  and right group-like projections assigned to  simple coideals. For more detailed exposition of what follows see \cite{Szulim_Master}. 

Let us first discuss semisimple  coideals assigned to right group-like projections which are not two sided. 
For every $\beta\in\CC^\times$ the element $P_\beta\in\cH_{n^2}$ defined by $P_\beta=\frac{1}{n}\sum_{k=0}^{n-1} (g+\beta x)^k$ is a non-degenerate right (but not left) group-like projection satisfying  \begin{equation}\label{DelP}\Delta(P_\beta) = \frac{1}{n}\sum_{k=0}^{n-1} S((g+\beta x)^k)g^{-1}\otimes (g+\beta x)^{n-k-1}.\end{equation}
In particular \begin{equation}\label{vi}\cV_{P_\beta}=\textrm{span}\{(g+\beta x)^k:k=0,1,\ldots n-1\},\end{equation}  see also \cite[
Proposition 4.6]{Chir_Kasp_Szu}. 
Using \Cref{thm_gcoin} and \Cref{DelP} we see that  $\dim L^{g^{-1}}_{\cV_{P_\beta}} = 1$ and  $\dim L^{g^k}_{\cV_{P_\beta}} = 0$ for $k\neq -1\mod n$. There exists   $g^{-1}$-cointegral  $\phi_{g^{-1}}\in L^{g^{-1}}_{\cV_{P_\beta}}$ such that \[(\id\otimes\phi_{g^{-1}})(\Delta({P_\beta})) = g^{-1}\] Using \Cref{DelP} we get \[\phi_{g^{-1}}((g+\beta x)^k) = n\delta_{k,n-1}\] where  $k\in\{0,1,\ldots,n-1\}$. 
Note that in the discussed case, the element  $z_{g^{-1}}\in\cV_{P_\beta}$ entering the proof of \Cref{thm_semi}  is equal $(g+\beta x)^{n-1} = (g+\beta x)^{-1} $ which is evidently an invertible element (c.f. the assumptions of  \Cref{thm_semi}). 

Next we shall discuss semisimple coideals assigned to two sided group-like projections. For every divisor $d$ of $n$ we have a Hopf subalgebra $H_d\subset \cH_{n^2}$   generated by  $g^d$, see \cite[Proposition 4.6]{Chir_Kasp_Szu}). Clearly, $H_d = \CC[G'] $ where $G'$ is the cyclic group of order $\frac{n}{d}$ (generated by $g^d$).  Defining  $P_d:=\frac{d}{n}\sum_{k=0}^{\frac{n}{d}-1}g^{dk}$ we can check that $P_d\in L^{H_d}_\varepsilon $. Observing that 
\[\Delta(P_d) =\frac{d}{n} \left(\sum_{k=0}^{\frac{n}{d}-1}g^{dk}\otimes g^{dk}\right)\] we see  that   $\dim L^{g^{j}}_{H_d} = 1$ if and only if $j =0\mod d$ (see  \Cref{thm_two_sided}).   More precisely  the system of vectors   $(g^{dm})_{m\in\{0,\ldots,\frac{n}{d}-1\}}$ is a  basis of $H_d$ and defining \[\phi_{g^{dk}}(g^{dm}) =\frac{n}{d} \delta_{m,k}\] where   $m\in\{0,\ldots,\frac{n}{d}-1\}$, we get a non-zero element of  $L^{g^{dk}}_{H_d}$.
It is easy to check that $z_{g^{dm}} = g^{dm}$ which is evidently an invertible element of $H_d$. 

Let us discuss  non-semisimple left coideal subalgebras of $\cH_{n^2}$.
For every divisor $d$ of $n$ we have a left coideal subalgebra $\cN_{d,x}$ generated by $x$ and $g^d$, see \cite[Proposition 4.3]{Chir_Kasp_Szu}) and defining  $\Lambda:=\sum_{k=0}^{\frac{n}{d}-1}g^{dk} x^{n-1}$ we can check that $\Lambda\in L^{\cN_{d,x}}_\varepsilon $. Observing that 
\[\Delta\left(\sum_{k=0}^{\frac{n}{d}-1}g^{dk}x^{n-1}\right) = \left(\sum_{k=0}^{\frac{n}{d}-1}g^{dk}\otimes g^{dk}\right) \left(\sum_{j=0}^{n-1} S(x^j)g^{-1}\otimes x^{n-j-1}\right)\] and using    \Cref{thm_two_sided} we see that   $\dim L^{g^{j}}_{\cN_{d,x}} = 1$ if and only if $j = -1\mod d$.   More precisely  the system of vectors   $(g^{dm}x^l)_{m\in\{0,\ldots,\frac{n}{d}-1\}, l\in\{0,n-1\}}$ is a  basis of $\cN_\Lambda$ and defining \[\phi_{g^{dk-1}}(g^{dm}x^l) = \delta_{l,n-1}\delta_{m,k}\]  where  $m\in\{0,\ldots,\frac{n}{d}-1\}$ and $l\in\{0,n-1\}$, we get a non-zero element of  $L^{g^{dk-1}}_{\cN_{d,x}}$.

\subsection*{Acknowledgments}
 PK was partially supported by the NCN (National Center of Science) grant
 2015/17/B/ST1/00085. 

I am grateful to Alex Chirvasitu for the discussions over the subject of this paper.


\bibliography{taft}{}
\bibliographystyle{plain}

\end{document}